\documentclass[12pt]{amsart}

\usepackage{amsmath,amssymb}

\newcommand{\E}{\ensuremath{\mathcal E}}
\newcommand{\F}{\ensuremath{\mathcal F}}
\newcommand{\GL}{\textup{GL}}
\newcommand{\GLK}{\textup{GL}_{\mathcal K}}

\newcommand{\I}{\ensuremath{\mathcal I}}
\newcommand{\K}{\ensuremath{\mathcal K}}
\newcommand{\bdd}{\ensuremath{\mathcal L}}
\newcommand{\N}{\ensuremath{\mathbb N}}
\renewcommand{\O}{\ensuremath{\mathcal O}}

\newcommand{\R}{\ensuremath{\mathbb R}}

\renewcommand{\S}{\ensuremath{\mathcal S \mathcal S}}

\newcommand{\into}{\hookrightarrow}

\renewcommand\span{\operatorname{span}}

\theoremstyle{plain}
\newtheorem{thm}{Theorem}

\newtheorem{cor}[thm]{Corollary}

\newtheorem{lem}[thm]{Lemma}

\newtheorem{prop}[thm]{Proposition}

\newtheorem{thm*}{Theorem}

  {\addvspace{\bigskipamount}\noindent{\bf Example\ }}%
  {}

\theoremstyle{definition}
\newtheorem{dfn}[thm]{Definition}

\begin{document}

\title[Fredholm $\Delta$-Filtrations of Fredholm Groups]
{Fredholm $\Delta$-Filtrations of Fredholm Groups}
\author[Haj Saeedi Sadegh and Jody Trout]{Ahmad Reza Haj Saeedi Sadegh and Jody Trout}
\date{\today}
\email{Ahmadreza.Hajsaeedisadegh@dartmouth.edu, jody.trout@dartmouth.edu}       
\address{6188 Kemeny Hall\\ 
         Dartmouth College\\
         Hanover, NH 03755}


\begin{abstract}
In this short note, we show that some well-known filtrations of infinite dimensional groups of Fredholm operators associated to certain perturbation classes are, in fact, Fredholm $\Delta$-filtrations (see Definition \ref{def:FredFilt} below.) For example, let $\E$ be a separable infinite dimensional real Hilbert space. The group $\GLK(\E)$ of all invertible operators on $\E$ which are compact perturbations of the identity is the structure group for Hilbert Fredholm manifolds and bundles  modelled on $\E$ \cite{ElwTr, Ksch, Mkhr}. Using an orthonormal basis, there are canonical inclusions of general linear groups:
$$\GL(1) \subset  \cdots \subset \GL(n) \subset \GL(n+1) \subset  \cdots \subset \GL(\infty) = \varinjlim \GL(n) \subset \GLK(\E).$$
We show this is a Fredholm $\Delta$-filtration of the Fredholm manifold $\GLK(\E)$ with dimension sequence $\Delta(n) = \dim(\GL(n)) = n^2$, which was not discussed in the classical Fredholm manifold literature because of the rigid constraint that the dimensions of a filtration increase only by one.
\end{abstract}

\maketitle

Let $\E$ be a separable infinite dimensional real Banach space. 
We will use the following notation: $\bdd(\E)$ denotes the Banach algebra of bounded linear operators on $\E$;
$\F(\E)$ denotes the bounded finite rank operators;
$\K=\K(\E)$ denotes the compact operators;
$\Phi(\E)$ denotes the Fredholm operators;
and $GL(\E)$ denotes the Banach-Lie group of units of $\bdd(\E)$,
with identity $I$. We let $\E^*$ denote the dual space of $\E$.

\begin{dfn}\label{D:perturbation}
A {\bf perturbation class} \cite{ElwTr} on $\E$ is a subspace $P(\E)$ of
$\bdd(\E)$ such that: 
\begin{itemize}
    \item[(1)] $\F(\E)\subseteq P(\E)$.
    \item[(2)] $P(\E)$ is an ideal in $\bdd(\E)$.
    \item[(3)] $\Phi(\E) + P(\E) = \Phi(\E)$. 
\end{itemize}
We call $P = P(\E)$ a {\bf Banach perturbation class} if it has a complete algebra norm such that the inclusion of $P$ into $\bdd(E)$ is continuous.
\end{dfn}

\noindent
The finite-rank operators $\F(\E)$ form a perturbation class. The closed ideal $\K(\E)$ of compact operators is a Banach perturbation class under the operator norm of $\bdd(\E)$, as are the closed ideals of strictly singular operators $\S(\E)$ and inessential operators $\I(\E)$. When $\E$ is a Hilbert space, for each $1 \leq q < \infty$, the Schatten ideal $P_q = P_q(\E)$, which is defined as the closure of $\F(\E)$ under the norm $\|\, T\,\|_q =\left( \text{Trace}(T^*T)^{q/2} \right)^{1/q}$, gives a Banach perturbation class called the $q$-th Schatten class. If $q=1$ one obtains the trace-class operators, and if $q=2$ the Hilbert-Schmidt operators.

Given a perturbation class $P$ on $\E$, we define the {\bf Fredholm group of $P$} to be the normal subgroup of $GL(\E)$ given by:
$$GL_P(\E) =  GL(\E) \cap (I+P(\E)) = \{ T \in \GL(\E) \;|\; T=I+K, \ K\in P(\E) \}.$$
For $1 \leq q < \infty$, we  abbreviate $GL_{P_q}(\E) = GL_q(\E)$. We topologize $GL_P(\E)$ by requiring that the bijection $GL_P(\E) \to \O \subset P(\E)$,  $I+K \mapsto K,$ be a  homeomorphism, where $\O$ is the open set of all $K$ in $P(\E)$ with $I+K$ invertible. The group $\GLK(\E)$ is the structure group for smooth Banach Fredholm manifolds \cite{Mkhr, Mkhr2} modelled on $\E$.

\begin{dfn}\label{def:flag} Let $\Delta=\{\delta(n)\}_{n=1}^\infty$ be a strictly increasing sequence of natural numbers. A {\bf $\Delta$-flag} \cite{SadeghTrout26} in a Banach space $\E$ is a sequence $\{E_n\}_{n=1}^\infty$ of finite-dimensional subspaces of $\E$ such that for all $n \geq 1:$
\begin{itemize}
    \item[a.)] $\dim(E_n) = \delta(n)$.
    \item[b.)] $E_n \subset E_{n+1}$.
    \item[c.)]  $\bigcup_{n=1}^\infty E_n$ is dense in $\E$.
    \item[d.)] $E_n$ has a closed complement $E^{\infty-n}$ in $\E$.
    \item[e.)] $E^{\infty - n} \supset E^{\infty - (n+1)}$.
  \end{itemize}
\end{dfn}
In the classical literature, a flag is an $\N$-flag where $\dim(E_n) = n$ for all $n \geq 1$. By relaxing this dimension constraint, we obtain a more flexible theory (see \cite{SadeghTrout26}.)

Let $\{E_n\}_1^\infty$ be an $\N$-flag for $\E$.  Then each  inclusion $E_n \subset E_n \oplus E^{\infty - n} \cong \E$ induces a canonical inclusion $$M(n)=_\text{def}  \bdd(E_n) \subset \F(\E)$$  given by $T \mapsto T \oplus 0$. That is, we identify $M(n)$ with the subalgebra of $\bdd(\E)$ of all operators that map $E_n \cong \R^n$ to itself and are zero on the complement $E^{\infty -n}$. Note that $\dim(M(n)) = n^2$. Similarly, the inclusions $E_n \subset E_{n+1}$ induce (via (e)) canonical inclusions $M(n) \subset M(n+1)$. Thus, for any perturbation class $P(\E)$, there are canonical inclusions:
$$M(1) \subset \cdots \subset M(n) \subset M(n+1) \subset \cdots \subset M(\infty) = \bigcup_{n=1}^\infty M(n) \subset P(\E).$$
\begin{dfn} A Banach perturbation class $P = P(\E)$ will be called {\bf matricial} if there is an $\N$-flag in $\E$ such that the induced sequence $\{M(n)\}_1^\infty$ is a $\Delta$-flag in the Banach space $P$ where $\Delta = \{n^2\}_1^\infty$.
\end{dfn}

For example, if $\E$ does not have the approximation property, then $\K(\E)$ and, hence, $\S(\E)$ and $\I(\E)$, are clearly not matricial perturbation classes.

Recall that a Schauder basis $\{e_i\}_1^\infty$ is called {\bf shrinking} \cite{James, Reth} if  $\|f\|_n \to 0$ as $n \to \infty$ for all $f \in \E^*$, where $\|f\|_n$ is the norm of $f$ restricted to the closure of the span of $\{ e_i : i \geq n\}$. Every orthonormal basis is shrinking. The standard unit bases for $c_0$ and $\ell^q$, $1 < q < \infty$, are all shrinking, but not for $\ell^1$.

\begin{prop}  $a.)$  If $\E$  has a shrinking Schauder basis, then $\K(\E)$ is matricial.

$b.)$ If $\E$ is a Hilbert space, then $P_q(\E)$ is matricial for all $1 \leq q < \infty$.
\end{prop}

\begin{proof}
   $a.)$ Let $\{e_i\}_1^\infty$ be a shrinking Schauder basis for $\E$. Let $\{E_n\}_1^\infty$ be the associated $\N$-flag where $E_n = \span \ \{e_1, \dots e_{n}\}$ and  $E^{\infty - n} = \overline{\span \ \{ e_i : i \geq n\}}.$ Since $\E$ has a Schauder basis, it has the approximation property; so there is an isomorphism of the compacts  $\K(\E)=\overline{\F(\E)} \cong \E^*\otimes_{\varepsilon} \E$, where $\otimes_{\varepsilon}$ denotes the injective tensor product \cite{Ryan}. Since the basis is shrinking, the dual sequence $\{e_i^*\}_{i=1}^{\infty}$ forms a Schauder basis for the dual space $\E^*$ \cite{Reth, James}. By  \cite[Proposition 4.25]{Ryan}, the sequence $\{e_i^*\otimes e_j\}_{i,j =1}^\infty$ (with the square ordering) is a Schauder basis for $\E^*\otimes_{\varepsilon}\E \cong \K(\E)$. We then can identify $$M(n) =  \span \ \{e_i^* \otimes e_j : 1 \leq i, j \le n\} $$
   from which the result now follows since $M(\infty) = \span \ \{e_i^* \otimes e_j\}$ is dense in $\K(\E)$.

   $b.)$ Let $\{e_i\}_{i=1}^{\infty}$ be any orthonormal basis for the Hilbert space $\E$ and denote by $P_n$ the orthogonal projection onto $E_n = \span \ \{e_1, \dots, e_n\}$. Note that $E^{\infty - n} = E_n^\perp$. Let $T\in P_q(\E)$. By replacing $T$ with $|T| = (T^*T)^{1/2}$ we may assume, without loss of generality, that $T$ is a positive operator. It suffices to show $P_nTP_n\to T$ in the $q$-Schatten norm as $n\to \infty$. 
   
   First, note that
   \[||T(I-P_n)||_q=\Big(\sum_{i=n}^{\infty}||Te_i||^q\Big)^{1/q}\to 0\]
   as $n\to \infty$. Also, since $T$ is self-adjoint, we have 
   $||(I-P_n)T||_q=||T(I-P_n)||_q\to 0.$
   Now, we compute that
   \[||P_nTP_n-T||_q\leq ||P_nTP_n-P_nT||_q+||P_nT-T||_q\]
   \[\leq ||P_n|| \ ||T(I-P_n)||_q+||(I-P_n)T||_q = 2 \|T(I-Pn)\|_q \to 0\]
   as $n\to \infty$, since $||P_n||=1$. The result now follows as desired. \end{proof}

If $M$ and $N$ are smooth Banach manifolds \cite{Lang}, then a smooth map $f : M \to N$ is called {\bf Fredholm} if for each $x \in M$ the differential $Df(x) : T_xM \to T_{f(x)}N$ is a Fredholm operator. If the Fredholm index of $Df(x)$ is a constant integer $m$ on $M$, then $f$ is said to be a {\bf Fredholm map of index $m$}. Of interest are the Fredholm maps of index zero. The following is adapted from Mukherjea \cite{Mkhr4, Mkhr3, Mkhr, Mkhr2}, who imposed the rigid dimension constraint that $\dim(M_n) = n$, and the more general Definition 3.3 of the authors \cite{SadeghTrout26} for various types of $\Delta$-filtrations of Banach and Fredholm manifolds. We allow that the dimensions in the filtration are merely strictly increasing. Note that Eells and Elworthy in their brief ICM note \cite{EeElw1} do not explicitly state that $\dim(M_n) = n$ but reference Mukherjea's thesis, which does require that constraint.

\begin{dfn}\label{def:FredFilt}
Let $M$ be a smooth Banach manifold. Let $\Delta=\{\delta(n)\}_{n=1}^\infty$ be a strictly increasing sequence of natural numbers.  A {\bf Fredholm $\Delta$-filtration} \cite{SadeghTrout26} of $M$ is a sequence $\{ M_n \}_{n=1}^\infty$ of smooth submanifolds of $M$ such that for all $n \geq 1$:
\begin{itemize}
\item[a.)] $\dim(M_n) = \delta(n)$.
\item[b.)] $M_n \subset M_{n+1}$.
\item[c.)] The embeddings $M_n \into M_{n+1}$ and
$M_n \into M$ have trivial normal bundles.
\item[d.)]  $M_\infty =  \bigcup_{n = 1}^\infty M_n$ is dense in $M$.
\item[e.)] If $M_\infty$ is given the inductive limit topology, the  inclusion $i_\infty : M_{\infty} \into M$ is a homotopy equivalence.
\item[f.)] There is a Fredholm map of index zero $f : M \to \E$ to a Banach space $\E$ that is transverse to a $\Delta$-flag $\{E_n\}_{n=1}^\infty$ of $\E$ for which $M_n = f^{-1}(E_n)$ for all $n \geq 1$.
\end{itemize}
\end{dfn}

If $\E \cong E \oplus F$ for two closed subspaces $E$ and $F$, we denote by $\GL(\E, E, F)$ the subgroup of $\GL(\E)$ consisting of those operators which map $E$ to itself and are the identity on $F$. If $\dim(E) < \infty$, then for any perturbation class $P = P(\E)$ there are canonical inclusions $\GL(E) \cong \GL(\E, E, F) \subset \GL_\F(\E) \subset \GL_P(\E).$
Thus, if $\{E_n\}_1^\infty$ is an $\N$-flag in $\E$, and we identify
$$\GL(n) =_\text{def}\GL(E_n) \cong \GL(\E, E_n, E^{\infty -n})$$
for all $n \geq 1$, there are canonical inclusions of subgroups:
{\small $$\GL(1) \subset \cdots \subset \GL(n) \subset \GL(n+1) \subset \cdots \subset \GL(\infty) = \bigcup_{n=1}^\infty \GL(n)  \subset \GL_\F(\E) \subset \GL_P(\E).$$}

The following lemma was proved by Palais \cite{Palais1} for the Banach perturbation classes of the compacts and Schatten classes on a separable Hilbert space. \v Svarc \cite{Svarc} proved it for the compacts on a large class of Banach spaces, including all those with a Schauder basis.  (See also the paper by \c Geba \cite{Geba}.) The general result is Theorem 1.3 \cite{ElwTr} of Elworthy and Tromba.\footnote{They state it was proved independently by Palais and Elworthy.} The proof uses the fact that the inclusion $\GL(\infty) \hookrightarrow \GL_P(\E)$ factors as the inclusions $\GL(\infty) \hookrightarrow \GL_\F(\E)$ and $\GL_\F(\E) \hookrightarrow \GL_P(\E)$, which are both homotopy equivalences via Whitehead arguments.

\begin{lem}[Generalized Palais-\v Svarc Lemma] For any perturbation class $P(\E)$ and any $\N$-flag in $\E$ the inclusion $\GL(\infty) \hookrightarrow \GL_P(\E)$ is a homotopy equivalence. \end{lem}
    
We now come to our main theorem.

\begin{thm} Let $P = P(\E)$ be a matricial Banach perturbation class on the Banach space $\E$.  The induced sequence $\{\GL(n)\}_1^\infty$ is a Fredholm $\Delta$-filtration of the Banach Fredholm  $P$-manifold $M = \GL_P(\E)$ with dimension sequence $\Delta = \{n^2\}_1^\infty$. \end{thm}

\begin{proof} The map $f : \GL_P(\E) \to P$ given by $f(T) = T-I$ identifies $\GL_P(\E)$ bijectively with the open subset $\O$ of the Banach space $P$ of all $K \in P$ such that $T = I + K$ is invertible, i.e., $T \in \GL(\E)$, since the inclusion $P \hookrightarrow \bdd(\E)$ is continuous and $\GL(\E)$ is open in $\bdd(\E)$. This gives $\GL_P(\E)$ the structure of a smooth Fredholm $P$-manifold (and, in fact, a Banach-Lie group \cite{PdlH2}) and makes $f : \GL_P(\E) \to \O$ a diffeomorphism, hence, a Fredholm map of index zero, that is transverse to each $M(n) \subset P$. It follows that for each $n$, $\GL(n) = f^{-1}(M(n))$ is a smooth submanifold of $\GL_P(\E)$ with $\dim(\GL(n)) = n^2$ and so conditions (a), (b) and (f) of Definition \ref{def:FredFilt} are clearly satisfied since $\{M(n)\}_1^\infty$ is a $\Delta$-flag in $P$. Condition (c) is satisfied since the pullback of a normal bundle is a normal bundle and $M(n) \subset M(n+1)$ and $M(n) \subset P$ have trivial normal bundles. Condition (d) is satisfied since $P$ is matricial. Finally, the homotopy equivalence (e) of the inclusion $\GL(\infty) \hookrightarrow \GL_P(\E)$ follows from the Generalized Palais-\v Svarc Lemma above.
\end{proof}

We thus have the following immediate corollaries using Proposition 4.

\begin{cor} Let $\E$ be a Banach space with a shrinking Schauder basis. Then the induced sequence of general linear groups $\{\GL(n)\}_1^\infty$ is a Fredholm $\Delta$-filtration of the Banach Fredholm  $\K$-manifold $M = \GL_\K(\E)$.
\end{cor}

\begin{cor} Let $\E$ be a Hilbert space with an orthonormal basis. The induced sequence of general linear groups $\{\GL(n)\}_1^\infty$ is a Fredholm $\Delta$-filtration of the Banach Fredholm  $P_q$-manifold $M = \GL_q(\E)$ for all $1 \leq q < \infty$.
\end{cor}

\bibliographystyle{elsarticle-num}
\bibliography{revised.bibliography}

\end{document}